\documentclass[12pt,reqno]{amsart}

\usepackage{amssymb}
\usepackage[all]{xy}

\newtheorem{theorem}{Theorem}[section]
\newtheorem{proposition}[theorem]{Proposition}
\newtheorem{lemma}[theorem]{Lemma}
\newtheorem{corollary}[theorem]{Corollary}

\theoremstyle{definition}
\newtheorem{definition}[theorem]{Definition}

\numberwithin{equation}{section}

\begin{document}

\baselineskip=15.5pt

\title[Principal bundles with holomorphic connections over a Calabi-Yau manifold]{Principal bundles with 
holomorphic connections over a K\"ahler Calabi-Yau manifold}

\author[I. Biswas]{Indranil Biswas}

\address{Department of Mathematics, Shiv Nadar University, NH91, Tehsil Dadri,
Greater Noida, Uttar Pradesh 201314, India}

\email{indranil.biswas@snu.edu.in, indranil29@gmail.com}

\author[S. Dumitrescu]{Sorin Dumitrescu}

\address{Universit\'e C\^ote d'Azur, CNRS, LJAD, France}

\email{dumitres@unice.fr}



\date{}

\begin{abstract}
We prove that any holomorphic vector bundle admitting a holomorphic connection, over a compact K\"ahler 
Calabi-Yau manifold, also admits a flat holomorphic connection. This addresses a particular case of a question 
asked by Atiyah and generalizes a result previously obtained in \cite{BD} for simply connected compact K\"ahler 
Calabi-Yau manifolds. We give some applications of it in the framework of Cartan geometries and foliated Cartan 
geometries on K\"ahler Calabi-Yau manifolds.
\end{abstract}

\maketitle

\section{Introduction}

A systematic study of holomorphic principal bundles admitting holomorphic connections over compact K\"ahler 
manifolds was initiated by Atiyah in \cite{At}. In particular, he proved that holomorphic vector bundles 
bearing a holomorphic connection over a compact K\"ahler manifold have vanishing real Chern classes of 
positive degree (see \cite[p.~197, Theorem 6]{At}). Therefore, Atiyah asked in \cite{At} whether all 
holomorphic vector bundles admitting a holomorphic connection over a compact K\"ahler manifold also admit a 
{\it flat } holomorphic connection. It may be mentioned that Murakami constructed examples of principal
compact torus bundles over compact torus admitting a holomorphic connection but not admitting any flat
holomorphic connection \cite{Mu1}, \cite{Mu2}. This shows that Atiyah's question can't be extended to
principal bundles with a complex Lie group as a structure group.

Notice that in the specific case where the holomorphic vector bundle is the holomorphic tangent space of a 
compact K\"ahler manifold $M$, Atiyah's question has a positive answer as a consequence of Yau's proof of 
the Calabi conjecture \cite{Ya}. Indeed, if the first real Chern class of a compact K\"ahler manifold
$M$ vanishes, Yau's proof of Calabi conjecture endows $M$ with a 
Ricci-flat metric \cite{Ya}, meaning $M$ is a so-called K\"ahler Calabi-Yau manifold. If the 
second Chern class also vanishes, this implies that the Ricci-flat metric is flat. As a consequence of 
Bieberbach's Theorem, $M$ has a finite unramified cover which is a compact complex torus. Therefore, up to 
a finite unramified covering, the holomorphic tangent bundle $TM$ is holomorphically trivial. This implies that 
$TM$ admits a holomorphic flat connection (with finite monodromy).

Recall that Beauville-Bogomolov decomposition Theorem \cite{Be,Bo} asserts that any K\"ahler Calabi-Yau 
manifold admits a finite unramified cover which is a product of a simply connected K\"ahler Calabi-Yau 
manifold and a compact complex torus.

In this article we positively answer Atiyah's above question for holomorphic vector bundles over K\"ahler 
Calabi-Yau manifolds. More precisely our main result (Theorem \ref{thm1}) proves that for any complex affine 
algebraic group $G$, all holomorphic principal $G$-bundles admitting a holomorphic connection over a compact 
K\"ahler Calabi-Yau manifold, also admit a holomorphic flat connection. Moreover, if the base
manifold is a compact {\it 
simply connected} K\"ahler Calabi-Yau manifold, this implies that the holomorphic principal $G$-bundle is 
holomorphically trivial and it admits a unique holomorphic connection which is flat since it coincides with 
the trivialization (see Corollary \ref{cor1}). This particular case was previously settled in \cite{BD}.

A second part of the article deals with applications to the classification of holomorphic Cartan geometries 
and foliated Cartan geometries on compact K\"ahler Calabi-Yau manifolds.

Recall that compact K\"ahler Calabi-Yau manifolds bearing a holomorphic Cartan geometry are known to be 
covered by compact complex tori \cite{BM,Du}. Moreover, on compact complex tori all holomorphic Cartan 
geometries are translation invariant \cite{BD4}. A foliated version of this result is proved here in Theorem 
\ref{fol Cartan} which states that simply connected compact K\"ahler Calabi-Yau manifolds do not admit 
foliated holomorphic Cartan geometries.

The more flexible notions of {\it branched} holomorphic Cartan geometries and of {\it generalized} Cartan 
geometries were introduced in \cite{BD,BD2}. It was proved in \cite{BD} that a compact simply connected 
K\"ahler Calabi-Yau manifold admits a branched Cartan geometry with model the complex projective space if 
and only if the manifold is projective. The proof is based on a result proving that on any compact simply 
connected K\"ahler Calabi-Yau manifold all branched and generalized holomorphic Cartan geometries are flat. 
We recall those results in Section \ref{sect4} and give a proof of this flatness result (see Corollary 
\ref{flatness}) based on Theorem \ref{thm1} and its Corollary \ref{cor1}.

The organization of this article is the following. The main Theorem \ref{thm1} and its Corollary \ref{cor1} 
are proved in Section \ref{sect2}, while an important technical key ingredient (Proposition \ref{prop1}) is 
proved in Section \ref{sect3}. Section \ref{sect4} presents the notions of branched and generalized Cartan 
geometries and deduce from Theorem \ref{thm1} and its Corollary \ref{cor1} the above flatness result for 
generalized Cartan geometries on compact simply connected K\"ahler Calabi-Yau manifolds (see Corollary 
\ref{flatness}). Section \ref{sect5} deals with foliated Cartan geometries and provides a proof of Theorem 
\ref{fol Cartan}.

\section{A criterion for a flat connection} \label{sect2}

Let $M$ be a compact connected K\"ahler manifold. Take a K\"ahler metric $h$ on $M$; let
$\omega_h$ denote the $(1,\, 1)$--form on $M$ corresponding to $h$.
Since $h$ is a K\"ahler metric, the Levi--Civita connection on $TM$ corresponding to $h$
coincides with the Chern connection on $TM$ corresponding to $h$. The first Chern form on $M$ for
$h$ coincides with the Ricci form for $h$. The K\"ahler metric $h$ is called a \textit{K\"ahler--Einstein}
metric if the Ricci form for $h$ is a constant scalar multiple of $\omega_h$.

Assume that the first Chern class $c_1(TM)\, \in\, H^{2}(M,\, {\mathbb R})$ is zero. Calabi conjectured
that given any K\"ahler class on $M$, there is a K\"ahler form on $M$ lying in this class whose Ricci
curvature vanishes identically. This conjecture was proved by Yau \cite{Ya}. Therefore the class of compact
K\"ahler manifolds with vanishing first Chern class is the same as the class of 
compact K\"ahler manifolds admitting a K\"ahler metric with vanishing Ricci curvature (also called Ricci-flat). The manifolds in this class are known as (K\"ahler) {\it Calabi-Yau manifolds}.

Let $M$ be a compact connected K\"ahler manifold equipped with a K\"ahler form $\omega$. 

Let $G$ be a connected complex affine algebraic group. Recall that an algebraic group $G$ is called complex 
affine if it admits a linear algebraic representation with discrete kernel (for $G$ simply connected, this 
is equivalent to the existence of an injective $\mathbb C$--linear representation of the Lie algebra 
$\mathfrak g$). For example, complex semi-simple Lie group are complex affine
(see \cite[Chapter XVII, Theorem 3.2]{Ho}) and simply connected complex algebraic groups are complex
affine by the classical Ado's Theorem \cite{Hu}.

The Lie algebra of $G$ will be denoted by $\mathfrak g$. Let
\begin{equation}\label{c1}
p\, \,:\,\, E_G\,\,\longrightarrow\,\, M
\end{equation}
be a holomorphic principal $G$--bundle on $M$. The action of $G$ on $E_G$ produces an action of $G$
to the holomorphic tangent bundle $TE_G$. The quotient
\begin{equation}\label{c2}
{\rm At}(E_G)\,\, :=\,\, (TE_G)/G\,\, \longrightarrow\,\, E_G/G \,\,=\,\, M
\end{equation}
is a holomorphic vector bundle on $M$, which is called the Atiyah bundle for $E_G$ \cite{At}.

Let
$$
T_{E_G/M} \, \longrightarrow\, M
$$
be the relative tangent bundle for the projection $p$ in \eqref{c1}, so $T_{E_G/M}$ is the kernel
of the differential $dp\, :\, TE_G \, \longrightarrow\, p^*TM$ of $p$. The above action of $G$ on $TE_G$
preserves the subbundle $T_{E_G/M}$, and the quotient $T_{E_G/M}/G$ is identified with the adjoint
vector bundle $\text{ad}(E_G)$ on $M$ associated to $E_G$ for the adjoint action of $G$ on its Lie algebra
$\mathfrak g$.

On $E_G$, we have the following short exact sequence of holomorphic vector bundles equipped with actions of $G$:
$$
0\, \longrightarrow\, T_{E_G/M} \, \longrightarrow\, TE_G \,
\stackrel{dp}{\longrightarrow}\, p^*TM\,\longrightarrow\, 0.
$$
Taking quotient of it, for the actions of $G$, we get the short exact sequence of holomorphic vector bundle 
\begin{equation}\label{c3}
0\, \longrightarrow\, \text{ad}(E_G) \, \longrightarrow\, \text{At}(E_G) \,
\stackrel{d'p}{\longrightarrow}\, TM\,\longrightarrow\, 0,
\end{equation}
where $d'p$ is induced by $dp$, and $\text{At}(E_G)$ is the Atiyah bundle in \eqref{c2}.

A holomorphic connection of $E_G$ is a holomorphic homomorphism of vector bundles
\begin{equation}\label{c4}
D\,:\, TM \, \longrightarrow\, \text{At}(E_G)
\end{equation}
such that
\begin{equation}\label{c5}
(d'p)\circ D\,\,=\,\, {\rm Id}_{TM},
\end{equation}
where $d'p$ is the projection in \eqref{c3}. Note that the Lie bracket operation of the sheaf of
holomorphic vector fields on $E_G$ produces a Lie algebra structure on the sheaf 
$\text{At}(E_G)$. This Lie algebra structure preserves the subbundle $\text{ad}(E_G)$ in \eqref{c3}.
Indeed, this follows immediately from the fact that the sheaf of vertical vector fields on $E_G$, for the
projection $p$ in \eqref{c1}, is closed under the operation of Lie bracket. Therefore, \eqref{c3} is a short
exact sequence of sheaves of Lie algebras.

Take a holomorphic connection $D$ on $E_G$ (as in \eqref{c4}). The curvature of $D$ is the obstruction
of $D$ to be a Lie algebra structure preserving homomorphism. More precisely, for locally defined
holomorphic vector fields $v$ and $w$ on $M$, define
$$
R(D)(v,\, w) \,=\, [D(v),\, D(w)] - D([v,\, w]).
$$
{}From \eqref{c5} it follows that $d'p R(D)(v,\, w)\,=\, 0$, and also
$$
R(D)(f_1v,\, f_2w) \,=\, f_1f_2\cdot R(D)(v,\, w)
$$
for any locally defined holomorphic functions $f_1$ and $f_2$. Consequently, we have
$$
R(D)\,\, \in\,\, H^0(M,\, \Omega^2_M\otimes\text{ad}(E_G)).
$$
The connection $D$ is called integrable (or flat) if $R(D)\,=\, 0$.

Let $E$ be a holomorphic vector bundle on $M$ of rank $r$. Let $$E_{\text{GL}(r)}\, \longrightarrow\, M$$
be the corresponding holomorphic principal $\text{GL}(r, {\mathbb C})$--bundle. We note that the
fiber of $E_{\text{GL}(r)}$ over any point $x\, \in\, M$ is the space of all linear isomorphisms from
${\mathbb C}^r$ to the fiber $E_x$. Giving a holomorphic connection on the principal $\text{GL}(r,
{\mathbb C})$--bundle $E_{\text{GL}(r)}$ is same as giving a holomorphic differential operator of
order one
$$
D\, :\, E\, \longrightarrow\, E\otimes\Omega^1_M
$$
such that $D(fs) \,=\, fD(s) + s\otimes df$ for all locally defined holomorphic section $s$ of $E$ and
all locally defined holomorphic function $f$ on $M$.

\begin{lemma}\label{lem1}
Let $E_G$ be a holomorphic principal $G$--bundle on a compact connected K\"ahler
manifold $M$ admitting a holomorphic connection.
Then all the characteristic classes of $E_G$ of positive degree, with coefficients in $\mathbb R$,
vanish.
\end{lemma}

\begin{proof}
Take any holomorphic homomorphism
$$
\rho\, :\, G\, \longrightarrow\, \text{GL}(W),
$$
where $W$ is a finite dimensional complex vector space. Let $E_W\,:=\, E_G\times^\rho W\,\longrightarrow\,
M$ be the holomorphic vector bundle on $M$ associated to $E_G$ for the $G$--module $W$. Any holomorphic
connection on $E_G$ induces a holomorphic connection on the associated vector bundle $E_W$. This
implies that the Chern class $c_i(E_W)\,\in\, H^{2i}(M,\, {\mathbb R})$ vanishes for every $i\, >\, 0$
\cite[p.~197, Theorem 6]{At}. Since this holds for all $G$--modules, we conclude that
every characteristic class of $E_G$ of positive degree, with coefficients in $\mathbb R$, vanishes.
\end{proof}

Let $M$ be a compact connected K\"ahler manifold of complex dimension $d$. Fix a K\"ahler form $\omega$
on $M$. The \textit{degree} of any torsionfree coherent analytic sheaf $V$ on $M$ is defined to be
$$
\text{degree}(V)\, :=\, \int_M c_1(\det V)\wedge\omega^{d-1} \, \in\, {\mathbb R}
$$
(see \cite[p.~168, (7.1)]{Ko}, \cite[p.~167, (6.15)]{Ko}), where $\det V$ is the determinant line bundle for
$V$ (see \cite[Ch.~V, \S~6]{Ko} for determinant bundle for a torsionfree coherent analytic sheaf). The
number
$$
\mu(V)\,\, :=\,\, \frac{\text{degree}(V)}{\text{rank}(V)}\,\,\in\, \, {\mathbb R}
$$
is called the \textit{slope} of $V$.

A torsionfree coherent analytic sheaf $V$ on $M$ is called \textit{stable} (respectively, \textit{semistable}) if
$$
\mu(F)\, <\, \mu(V)\, \ \ \text{ (respectively,}\ \mu(F)\, \leq \, \mu(V)\text{)}
$$
for every coherent analytic subsheaf $F\, \subset\, V$ with $0\, <\, \text{rank}(F)\, <\, \text{rank}(V)$
(see \cite[p.~168]{Ko}). A torsionfree coherent analytic sheaf $V$ on $M$ is called \textit{polystable}
if it is a direct sum of stable sheaves of same slope.

A holomorphic vector bundle $V$ on $M$ is called \textit{pseudostable} if there is a filtration of holomorphic
subbundles
$$
0\, =\, F_0\, \subset\, F_1\, \subset\, F_2\, \subset\, \cdots \, \subset\, F_{\ell-1}\, \subset\, F_\ell
\,=\, F
$$
such that 
\begin{enumerate}
\item the vector bundle $F_i/F_{i-1}$ is stable for all $1\, \leq\, i\, \leq\, \ell$, and

\item $\mu(F_1)\,=\, \mu(F_2/F_1)\,=\, \cdots\,=\, \mu(F_{\ell}/F_{\ell-1})$.
\end{enumerate}
(See \cite[p.~23, Definition 2.1]{BG}.)

\begin{proposition}\label{prop1}
Let $M$ be a compact connected K\"ahler manifold equipped with a Ricci flat K\"ahler form $\omega$.
Let $E$ be a holomorphic vector bundle on $M$ admitting a holomorphic connection. Then $E$
is pseudostable.
\end{proposition}

Proposition \ref{prop1} will be proved in the next section.

\begin{theorem}\label{thm1}
Let $M$ be a compact connected K\"ahler manifold equipped with a Ricci flat K\"ahler form $\omega$.
Let $E_G$ be a holomorphic principal $G$-bundle over $M$ admitting a holomorphic connection. Then $E_G$ admits a
flat holomorphic connection.
\end{theorem}

\begin{proof}
Let $R_u(G)\, \subset\, G$ be the unipotent radical of $G$. So the Levi quotient
$$
L(G)\,=\, G/R_u(G)
$$
is reductive \cite[p.~125]{Hu}. Let
$$
q\, :\, G\, \longrightarrow\, L(G)
$$
be the quotient map. Let
$$
E_{L(G)}\,\,:=\, E_G\times^q L(G)\,\, \longrightarrow\,\, M
$$
be the holomorphic principal $L(G)$--bundle on $M$ obtained by extending the structure group
of $E_G$ using $q$.

Take a holomorphic connection $D$ on the principal $G$--bundle $E_G$. The holomorphic connection on the
principal $L(G)$--bundle induced by $D$ will be denoted by $D_L$. The holomorphic connection on the
adjoint vector bundle $\text{ad}(E_{L(G)})$ induced by $D_L$ will be denoted by $\widehat{D}_L$.
Now from Proposition \ref{prop1} we know that the adjoint vector bundle $\text{ad}(E_{L(G)})$ is
pseudostable.

{}From Lemma \ref{lem1} we know that all the characteristic classes of $E_G$ of positive degree, with 
coefficients in $\mathbb R$, vanish.

Now from \cite[p.~20, Theorem 1.1]{BG} it follows that the principal $G$--bundle $E_G$ admits a
flat holomorphic connection.
\end{proof}

Notice that in Theorem \ref{thm1} the condition on $G$ to be a complex affine algebraic group is essential. 
Indeed, Proposition 4.3 in \cite{BD2} constructs an elliptic bundle over a $K3$ surface which admits a 
holomorphic connection but does not admit any holomorphic flat connection.

Observe that the main result in \cite{BD} (Theorem 6.2) is the following direct consequence of Theorem 
\ref{thm1}:

\begin{corollary}\label{cor1}
Let $M$ be a compact simply connected K\"ahler manifold equipped with a Ricci flat K\"ahler form $\omega$. 
Let $E_G$ be a holomorphic principal $G$-bundle over $M$ admitting a holomorphic connection. Then $E_G$ is 
the trivial holomorphic principal $G$-bundle over $M$, and the holomorphic connection on it coincides with 
the trivialization.
\end{corollary} 

\begin{proof}
By Theorem \ref{thm1}, $E_G$ admits a holomorphic flat connection. Since $M$ is simply connected, this 
implies $E_G$ is the trivial holomorphic principal $G$-bundle over $M$. Any connection on the trivial bundle 
is determined, with respect to the trivialization, by a holomorphic section of $\Omega^1_M \otimes {\mathfrak 
g}$. But compact simply connected K\"ahler manifolds do not admit nontrivial holomorphic $1$-forms. Therefore, 
the only holomorphic section of the holomorphic vector bundle $\Omega^1_M \otimes {\mathfrak g}$ is the 
trivial one. The corresponding holomorphic connection on the trivial holomorphic $G$-bundle identifies with 
the trivialization.
\end{proof} 

\section{Proof of Proposition \ref{prop1}} \label{sect3}

Since $\omega$ is a Ricci-flat K\"ahler form on $M$, it follows that the K\"ahler form $\omega$
is Hermitian--Einstein. L\"ubke proved that any vector bundle on a compact K\"ahler manifold admitting
an Hermitian--Einstein structure is polystable \cite[p.~245, THEOREM]{Lu} (see also
\cite[p.~177--178, Theorem 8.3]{Ko}). Therefore, the tangent bundle $TM$ is polystable with respect
to any Ricci-flat K\"ahler structure on $M$.

Let $E$ be a holomorphic vector bundle on $M$ equipped with a holomorphic connection $D$.

Since $TM$ is polystable with $c_1(TM)\,=\, 0$, and $E$ admits a holomorphic
connection, it follows that the vector bundle $E$ is semistable \cite[p.~2830]{Bi}.

A coherent analytic sheaf $W$ on $M$ is called reflexive if the natural homomorphism
$W\, \longrightarrow\, W^{**}$ is an isomorphism \cite[p.~151]{Ko}.

Let $V\, \subset\, E$ be a polystable subsheaf such that
\begin{itemize}
\item $\text{degree}(V)\,=\, 0$, and

\item the quotient $E/V$ is torsionfree.
\end{itemize}
The second condition implies that $V$ is reflexive (see \cite[p.~153, Proposition 4.13]{Ko}). Since
$\text{degree}(V)\,=\, 0\,=\, \text{degree}(E)$, it follows that $\text{degree}(E/V)\,=\, 0$. From
the given conditions that $E$ is semistable and $V$ is semistable with 
$\text{degree}(V)\,=\, 0\,=\, \text{degree}(E)$, it can be deduced that $E/V$ is semistable. To see
this, take any coherent analytic subsheaf $W\, \subset\, E/V$. Let $E'\, \subset\, E$ be the inverse
image of $W$ for the quotient map $E\, \longrightarrow\, E/V$. Then $E'$ fits in the following
short exact sequence of coherent analytic sheaves:
$$
0\, \longrightarrow\, V\, \longrightarrow\, E'\, \longrightarrow\, W \, \longrightarrow\, 0.
$$
Since $E$ is semistable of degree zero, we have $\text{degree}(E')\, \leq\, 0$. Since we also have
$\text{degree}(V)\,=\, 0$, from the above exact sequence it follows that $\text{degree}(W)\, \leq\, 0$.
Hence $E/V$ is semistable.

Let $d$ be the complex dimension of $M$. The ranks of $V$ and $E/V$ will be denoted
by $m$ and $n$ respectively. 
Since $V$ and $E/V$ are semistable, we have the Bogomolov inequality
\begin{equation}\label{g1}
((2m\cdot c_2(V)-(m-1)c_1(V)^2)\cup \omega^{d-2})\cap [M]\, \geq\, 0\, ,
\end{equation}
\begin{equation}\label{g2}
((2n\cdot c_2(E/V)-(n-1)c_1(E/V)^2)\cup \omega^{d-2})\cap [M]\, \geq\, 0
\end{equation}
\cite[Lemma 2.1]{BM}.

We will show that the inequalities in \eqref{g1} and \eqref{g2} are equalities.
Denote the sheaf $E/V$ by $W$. We have
$$
2(m+n)c_2(V\oplus W)- (m+n-1)c_1(V\oplus W)^2
$$
$$
=\,
2(m+n)(c_2(V)+c_2(W)+c_1(V)c_1(W))- (m+n-1)(c_1(V)^2+c_1(W)^2+2c_1(V)c_1(W))
$$
$$
=\, \frac{m+n}{m}(2mc_2(V)-(m-1)c_1(V)^2)+\frac{m+n}{s}(2mc_2(W)-(n-1)c_1(W)^2)
$$
$$ 
-\frac{n}{m}c_1(V)^2-\frac{m}{n}c_1(W)^2 + 2c_1(V)c_1(W)
$$
$$
\,=\,\frac{m+n}{m}(2mc_2(V)-(m-1)c_1(V)^2)+\frac{m+n}{n}(2mc_2(W)-(n-1)c_1(W)^2) 
$$
$$
-\, \frac{1}{mn}(n\cdot c_1(V)-m\cdot c_1(W))^2\, .
$$
On the other hand, $c_i(V\oplus W)\,=\, c_i(E)\,=\, 0$, so
$$
\frac{m+n}{m}((2m\cdot c_2(V)-(m-1)c_1(V)^2)\cup \omega^{d-2})\cap [M]
$$
$$
+\, \frac{m+n}{n}((2mc_2(W)-(n-1)c_1(W)^2)\cup \omega^{d-2})\cap [M]
$$
\begin{equation}\label{g3}
-
\frac{1}{mn}((n\cdot c_1(V)-m\cdot c_1(W))^2 \cup \omega^{d-2})\cap [M]
\,=\,0\, .
\end{equation}
We have
$$
(c_1(V)\cup \omega^{d-1})\cap [M]\,=\, 0\,=\, (c_1(W)\cup \omega^{d-1})\cap [M]
$$
because $\text{degree}(V)\,=\, 0\,=\, \text{degree}(W)$. Hence
$$
((n\cdot c_1(V)-m\cdot c_1(W))\cup \omega^{d-1})\cap [M] \, =\, 0\, .
$$
Therefore, from Hodge index theorem (see \cite[Section 6.3]{Vo}) it follows that
$$
((n\cdot c_1(V)-m\cdot c_1(W))^2 \cup \omega^{d-2})\cap [M]
\, \leq\, 0\, .
$$
In other words, in \eqref{g3},
$$
-\frac{1}{mn}((n\cdot c_1(V)-m\cdot c_1(W))^2 \cup \omega^{d-2})\cap [M]
\, \geq\, 0\, .
$$
Therefore, from \eqref{g3} we conclude that the inequalities in \eqref{g1} and \eqref{g2}
are equalities.

Since $((2m\cdot c_2(V)-(m-1)c_1(V)^2)\cup \omega^{d-2})\cap [M]\, =\, 0$, from
\cite[p.~40, Corollary~3]{BS} we conclude that $V$ is a polystable {\it vector bundle}
admitting a projectively flat unitary connection. Therefore, the projective bundle $P(V)$
for $V$ is given by a representation of $\pi_1(M)$ in $\text{PU}(m)$.

First assume that $V$ is preserved by the connection $D$ on $E$. Then $V$ is a 
subbundle of $E$, and the quotient $E/V$ has a holomorphic connection $D_1$ induced 
by $D$. Consequently, we may substitute $(E/V,\, D_1)$ in place in $(E,\, D)$ in the
above arguments. Note that we now have
\begin{equation}\label{r1}
\text{rank}(E/V,\, D_1)\,\, <\,\, \text{rank}(E,\, D).
\end{equation}

We will now show that if $V$ is \textit{not} preserved by the connection $D$ on $E$, then also
the rank can be lowered.

Assume that $V$ is not preserved by the connection $D$ on $E$. Therefore, the following composition of
homomorphisms
\begin{equation}\label{r2}
V\,\hookrightarrow\, E \, \stackrel{D}{\longrightarrow}\, E\otimes\Omega^1_M
\, \xrightarrow{\,\,\, q_V\otimes {\rm Id}_{\Omega^1_M}\,\,\,}\, (E/V)\otimes\Omega^1_M,
\end{equation}
where $q_V\, :\, E\, \longrightarrow\, E/V$ is the quotient map, is nonzero. The composition
of homomorphisms in \eqref{r2} is ${\mathcal O}_M$--linear, and hence it gives a nonzero section
$$
S_V\, \in\, H^0(M,\, \text{Hom}(V,\, E/V)\otimes \Omega^1_M).
$$ 
Let
\begin{equation}\label{de1}
\Phi\, :\, V\otimes TM\, \longrightarrow\, E/V
\end{equation}
be the homomorphism given by this section $S_V$. Since $V$ and $TM$ are polystable vector bundles
of degree zero, it follows that $V\otimes TM$ is also a polystable vector bundle of degree zero.
Indeed, both $V$ and $TM$ admits Hermitian--Einstein connections as they are polystable. Now,
the connection on $V\otimes TM$ induced by Hermitian--Einstein connections on $V$ and $TM$ is
also Hermitian--Einstein. Hence $V\otimes TM$ is polystable. The degree of $V\otimes TM$ is zero
because both $V$ and $TM$ are of degree zero. On the other hand, $E/V$ is torsionfree
and it is semistable of degree zero. Hence ${\rm Image}(\Phi)$ is a polystable vector bundle of
degree zero, where $\Phi$ is the homomorphism in \eqref{de1}.

Since ${\rm Image}(\Phi)$ is polystable vector bundle of degree zero, and $E/V$ is a torsionfree
semistable of degree zero, it follows that the quotient
\begin{equation}\label{d2}
(E/V)/{\rm Image}(\Phi)
\end{equation}
is torsionfree semistable sheaf of degree zero.

If $\text{rank}({\rm Image}(\Phi))\,=\, \text{rank}(E/V)$, then it can be shown
that $${\rm Image}(\Phi)\,=\, E/V.$$ To see this, consider the subsheaf
\begin{equation}\label{r5}
V_1\, :=\, q^{-1}_V({\rm Image}(\Phi)) \, \subset\, E,
\end{equation}
where $q_V$ is the projection in \eqref{r2}. Assume that $\text{rank}({\rm Image}(\Phi))\,=\, \text{rank}(E/V)$.
Then the inclusion map
\begin{equation}\label{r6}
\iota\, :\, V_1\, \longrightarrow\, E
\end{equation}
is an isomorphism over a nonempty open subset of $M$. Consider the homomorphism
$$
\bigwedge\nolimits^r \iota\, :\, \det (V_1)\,=\, \bigwedge\nolimits^r V_1\, \longrightarrow\,
\bigwedge\nolimits^r E\,=\, \det (E),
$$
where $r\,=\, \text{rank}(E)$ and $\iota$ is the homomorphism in \eqref{r6}. It gives a section
$$
\widetilde{\iota}\, \in\, H^0(M,\, \det (E)\otimes \det (V_1)^*).
$$
Let
$$
{\mathbb D}\,\,:=\,\, \text{Div}(\widetilde{\iota})
$$
be the divisor of the section $\widetilde{\iota}$. Since the divisor $\mathbb D$ is effective, we
have
$$
\text{degree}({\mathcal O}_M({\mathbb D})) \,\geq\, 0,
$$
and
\begin{equation}\label{r7}
\text{degree}({\mathcal O}_M({\mathbb D})) \,=\, 0
\end{equation}
if and only if ${\mathbb D}\,=\, 0$. But
$$
\text{degree}({\mathcal O}_M({\mathbb D}))\,=\, \text{degree}(\det (E)\otimes \det (V_1)^*)
\,=\, \text{degree}(E)-\text{degree}(V_1)\,=\,0,
$$
and hence \eqref{r7} holds. Consequently, ${\mathbb D}\,=\, 0$ and therefore $\iota$ in \eqref{r6}
is an isomorphism.

Note that
$$
(E/V)/{\rm Image}(\Phi)\,=\, E/V_1
$$
(see \eqref{d2} and \eqref{r5}), and hence $E/V_1$ is torsionfree semistable sheaf of degree zero.

For $V_1$ in \eqref{r5}, consider the composition of homomorphisms
$$
V_1\, \hookrightarrow\, E \, \stackrel{D}{\longrightarrow}\, E\otimes\Omega^1_M.
$$
This composition of homomorphisms maps $V\, \subset\, V_1$ to $V_1\otimes \Omega^1_M$, and hence
it produces a homomorphism
$$
V_1/V \, \longrightarrow\, (E/V_1)\otimes\Omega^1_M.
$$
This in turn defines a homomorphism
\begin{equation}\label{r3}
\Phi_1\, :\, (V_1/V)\otimes TM \, \longrightarrow\, E/V_1 .
\end{equation}

Now we have two possibilities, as before. In the first case, $\Phi_1\,=\, 0$, where $\Phi_1$ is the
homomorphism in \eqref{r3}. This implies that the holomorphic connection $D$ on $E$ preserves $V_1$.
Consequently, $V_1$ is a subbundle of $E$ (see, for instance, Lemma 4.5 in \cite{BD2}). This implies that $V$ is a subbundle of $E$. So both
$V$ and $V_1/V\,=\, {\rm Image}(\Phi)$ are polystable vector bundles of degree zero.
Now we can replace the pair $(E,\, D)$ by the pair $(E/V_1,\, {\mathcal D})$, where
$\mathcal D$ is the holomorphic connection on $E/V_1$ induced by $D$. As in \eqref{r1}, we have
$$
\text{rank}(E/V_1,\, {\mathcal D})\,\, <\,\, \text{rank}(E,\, D).
$$

In the second case, $\Phi_1\, \not=\, 0$, where $\Phi_1$ is the
homomorphism in \eqref{r3}. Since both $TM$ and $V_1/V\,=\, {\rm Image}(\Phi)$ are polystable vector
bundles of degree zero, we conclude that $(V_1/V)\otimes TM$ is also a polystable vector bundle
of degree zero. Since $E/V_1$ is semistable of degree zero, and $(V_1/V)\otimes TM$ is also a polystable
vector bundle of degree zero, we conclude that ${\rm image}(\Phi_1)$ is a polystable vector
bundle of degree zero.

Let $q_1\, :\, E\, \longrightarrow\, E/V_1$ be the quotient map. Consider the subsheaf
$$
V_2\, :=\, q^{-1}_1({\rm image}(\Phi_1))\, \subset\, E.
$$
Repeat the arguments with $V_1$ for $V_2$. This way we get, inductively, a filtration of subbundles of $E$
$$
V_0\,=\, V\, \subset \, V_1 \subset\, V_2\, \subset\, \cdots \, \subset\,
V_{\ell-1}\, \subset\, V_\ell \,=\, V
$$
such that 
\begin{enumerate}
\item the vector bundles $V$ and $V_i/V_{i-1}$, $1\,\leq\, i\, \leq\, \ell$,
are all stable, and

\item $\mu(V)\,=\, \mu(V_1/V)\,=\, \mu(V_2/V_1)\,=\, \cdots\,=\, \mu (V_{\ell}/V_{\ell-1})$.
\end{enumerate}
Hence $E$ is pseudostable.
This completes the proof of Proposition \ref{prop1}.

\section{Cartan geometries on Calabi-Yau manifolds} \label{sect4}

A Cartan geometry is a geometric structure infinitesimally modelled on a homogeneous space.

Let $G$ be a connected complex Lie group and $H\, \subset\, G$ a connected
complex Lie subgroup. The Lie algebras of $H$ and $G$ will be denoted by
$\mathfrak h$ and $\mathfrak g$ respectively.

The following classical definition of a Cartan geometry is an infinitesimal generalization of the standard 
fibration $G \,\longrightarrow\, G/H$ seen as a principal $H$-bundle and endowed with the left-invariant 
Maurer-Cartan form of $G$ (see \cite{Sh}). We give the definition in the complex analytic category.

\begin{definition}\label{def}
A holomorphic Cartan geometry of type $(G,H)$ on a {complex} manifold $M$ is a holomorphic principal
$H$--bundle $\pi \,: E_H\, \longrightarrow\, M$ endowed with 
 a $\mathfrak g$--valued holomorphic $1$--form $\eta$ on $E_H$ satisfying the following
conditions:
\begin{enumerate}
\item $\eta\, :\, TE_H\, \longrightarrow\, E_H\times{\mathfrak g}$ is an
isomorphism.

\item $\eta$ is $H$--equivariant with $H$ acting on $\mathfrak g$ via the adjoint representation.

\item the restriction of $\eta$ to each fiber of $\pi$ coincides with the Maurer--Cartan form
associated to the action of $H$ on $E_H$.
\end{enumerate}

The Cartan geometry is called {\it flat} if its curvature $K(\eta)\,=\, d \eta+ \frac{1}{2} 
\lbrack \eta,\, \eta \rbrack$ vanishes.
\end{definition}

Note that Definition \ref{def} implies that the holomorphic tangent bundle of $E_H$ is parallelizable; this 
parallelization is defined by the $\eta$-constant holomorphic sections of $TE_H$. The family of 
$\eta$-constant holomorphic sections of $TE_H$ with values in the Lie subalgebra $\mathfrak h$ trivializes 
the relative tangent bundle of the fibration $\pi$. When restricted to the relative tangent bundle
for $\pi$, the form 
$\eta$ is a Lie algebra isomorphism defined on the family of $\eta$-constant vector fields and with values 
in $\mathfrak h$. Moreover, $\eta$ defines a Lie algebra isomorphism from the family of $\eta$-constant 
vector fields to $\mathfrak g$ if and only if the Cartan geometry defined by $\eta$ is flat in the sense of 
Definition \ref{def}.

Holomorphic affine and projective connections, holomorphic conformal structures and holomorphic Riemannian 
metrics can be defined as holomorphic {\it Cartan geometries} \cite{Sh}.

A more general notion of {\it a branched Cartan geometry} (invariant by pull-back through branched 
holomorphic maps) was introduced in \cite{BD}, as a higher dimensional (non flat) generalization of 
Mandelbaum's definition of branched complex affine and complex projective structures on Riemann surfaces.
Moreover an even more general definition of {\it a generalized Cartan geometry} (invariant by the pull-back 
through any holomorphic map) was introduced and worked out in \cite{BD2} (see also \cite{AM}).

\begin{definition} \label{general}
A generalized holomorphic Cartan geometry of type $(G,\,H)$ on a {complex} manifold $M$ is a holomorphic
principal $H$--bundle $\pi \,: E_H\, \longrightarrow\, M$ endowed with 
a $\mathfrak g$--valued holomorphic $1$--form $\eta$ on $E_H$ satisfying the following conditions:
\begin{enumerate}

\item $\eta$ is $H$--equivariant with $H$ acting on $\mathfrak g$ via the adjoint representation.

\item the restriction of $\eta$ to each fiber of $\pi$ coincides with the Maurer--Cartan form
associated to the action of $H$ on $E_H$.
\end{enumerate}

If $\eta \, :\, TE_H\, \longrightarrow\, E_H\times{\mathfrak g}$ is an isomorphism over a dense open set in 
$M$, the generalized Cartan geometry is a called branched Cartan geometry.

The generalized Cartan geometry is called {\it flat} if its curvature $K(\eta)\,=\, d \eta+ \frac{1}{2} 
\lbrack \eta,\, \eta \rbrack$ vanishes.
\end{definition}

Let us now give an equivalent bundle theoretical definition of generalized Cartan geometries (see also 
\cite{BD,BD2}).

Let us first define 
\begin{equation}\label{e2}
E_G\, \,:=\,\, E_H\times^H G \,\, \stackrel{\pi_G}{\longrightarrow}\,\, M
\end{equation}
as being the holomorphic principal $G$--bundle over $M$ constructed by extending the
structure group of $E_H$ using the inclusion of $H$ in $G$. More precisely, $E_G$ is the
quotient of $E_H\times G$ where two points $(c_1,\, g_1),\, (c_2,\, g_2)\, \in\,
E_H\times G$ are identified if there is an element $h\, \in\, H$ such that
$c_2\,=\, c_1h$ and $g_2\,=\, h^{-1}g_1$. The projection $\pi_G$ in \eqref{e2} is induced
by the following projection map $E_H\times G\, \longrightarrow\, M$, $(c,\, g)\,\longmapsto\, \pi(c)$, where $\pi : E_H \longrightarrow M$ is the projection map for the bundle $E_H$.
The action of $G$ on $E_G$ is induced by the
action of $G$ on $E_H\times G$ given by the right--translation action of $G$ on itself (as second factor).

Denote by ${\rm ad}(E_G)$ the adjoint bundle for $E_G$.

Equivalently a generalized holomorphic Cartan geometry of type $(G,H)$ is defined by a homomorphism 

$$\overline{\eta} \,:\,{\rm At}(E_H) \,\longrightarrow \,{\rm ad}(E_G)$$ from the Atiyah
bundle for $E_H$ to the adjoint bundle for $E_G$ making the following diagram of holomorphic vector bundles over
$M$ commutative 
\begin{equation}\label{e6a}
\begin{matrix}
0 & \longrightarrow & {\rm ad}(E_H) & {\longrightarrow} & {\rm At}(E_H)
& \stackrel{d'p}{\longrightarrow} & TM & \longrightarrow & 0\\
&& \Vert && ~\Big\downarrow\overline{\eta} && ~\Big\downarrow \phi\\
0 & \longrightarrow & {\rm ad}(E_H) & \stackrel{\iota_1}{\longrightarrow} & {\rm ad}(E_G)
& \longrightarrow & {\rm ad}(E_G)/{\rm ad}(E_H) & \longrightarrow & 0
\end{matrix}
\end{equation}
where the first line is the exact sequence defined in (\ref{c3}), while the second line is given by the canonical embedding $\iota_1$ of the adjoint bundle for $E_H$ into the adjoint bundle for $E_G$.
The homomorphism $\phi$ in \eqref{e6a}
is induced by $\overline{\eta}$.

In the classical case the homomorphism $\overline{\eta}$ is required to be an isomorphism (which is 
equivalent for $\phi$ in (\ref{e6a}) to be an isomorphism), while in the branched case $\overline{\eta}$ is 
required to be an isomorphism over a dense open set in $M$ (in which case $\phi$ in (\ref{e6a}) is an 
isomorphism over the same dense open set in $M$). In the classical and in the branched case this implies 
that the homogeneous space model $G/H$ has the same dimension as the manifold $M$. This is not required 
anymore in the generalized case.

It was proved in \cite{BD,BD2} that a generalized Cartan geometry $\overline{\eta}$ defines a unique 
holomorphic connection on the principal $G$-bundle $E_G$ which is flat if and only if the Cartan geometry 
is flat (see also \cite{Sh} for the classical case).

Therefore Theorem \ref{thm1} has the following consequence (see also Corollary 4.6 in \cite{BD2}).

\begin{corollary}\label{flatness}
Let $M$ be a compact simply-connected Calabi-Yau manifold endowed with a generalized holomorphic Cartan 
geometry $(E_H,\, \eta)$ modelled on $(G,H)$, with $G$ being a complex affine algebraic group. Then $(E_H, 
\eta)$ is flat. If $\eta$ is a branched Cartan geometry, $G/H$ must be compact and the branching set must 
be nontrivial.
\end{corollary}

Indeed, the holomorphic connection defined by the Cartan geometry $(E_H, \eta)$ on the extended principal 
$G$-bundle bundle $E_G$ must be flat by Theorem \ref{thm1}. Hence the Cartan geometry $(E_H,\, \eta)$ is 
flat. Since $M$ is simply connected, a flat geometry on $M$ with model $(G,H)$ is determined by a 
holomorphic map ${\rm dev} : M \longrightarrow G/H$ called the developing map (see \cite{BD,BD2,Sh}). In 
the branched case $G/H$ has the same dimension as $M$ and ${\rm dev}$ is a dominant map. This implies the 
developing map is surjective and $G/H$ is compact. The branching set must be nontrivial: if not $M$ would 
be a cover of $G/H$, but $G/H$ does not have trivial first Chern class (as $M$ does).

It is an open question to classify compact complex homogeneous spaces $G/H$ which admit a finite ramified 
cover which is a K\"ahler Calabi-Yau manifold $M$. In the case where $G/H$ is the complex projective space, 
the manifold $M$ must be Moishezon and being K\"ahler it is projective. Conversely, any projective manifold 
admits a holomorphic finite map onto the complex projective space (see \cite{BD,BD2} for more details).
The question of classification of flag manifolds admitting a finite ramified cover with trivial first Chern 
class is still open.

It should be mentioned that it was proved in \cite{BD4} that Cartan geometries on complex tori are 
translation invariant. They are not necessarily flat. It was proved in \cite{BD5} that generic 
(translation invariant) Cartan geometries with model the complex projective space on compact complex tori 
are not flat. Theorem \ref{thm1} goes into the direction to prove that K\"ahler Calabi-Yau manifolds 
admitting generalized Cartan geometries with model $(G,H)$, also admit flat generalized Cartan geometries 
with model $(G,H)$.

\section{Foliated Cartan geometries on Calabi-Yau manifolds}\label{sect5}

Here $M$ is a compact K\"ahler Calabi-Yau manifold endowed with a nonsingular holomorphic foliation 
$\mathcal F$.

In \cite[Section 2.2]{BD1} we defined a notion of transversal holomorphic (branched) Cartan geometry to a 
holomorphic foliation $\mathcal F$. Roughly speaking, this means that local transversals to the foliation 
are endowed with compatible geometric structures which are infinitesimally modelled on a given homogeneous 
space $G/H$, with $G$ complex connected Lie group and $H$ complex connected Lie subgroup in it. There 
exists a natural tensor of curvature and when this tensor vanishes, the transversal Cartan geometry is 
flat, meaning that local transversals to the foliation are locally modelled on the homogeneous space $G/H$ 
with transition maps given by the canonical (holomorphic) $G$-action on $G/H$. In this case the pull-back 
of $\mathcal F$ to the universal cover $\tilde{M}$ of $M$ is defined by a submersion $\tilde M 
\longrightarrow G/H$ called the developing map of the transverse flat Cartan geometry (see section 3.2 in 
\cite{BD1}). The more general notion of transversal generalized Cartan geometry was worked out in 
\cite{BD3}.

Let us present the bundle theoretic definition of a transversal generalized Cartan geometry (as defined in 
\cite{BD1,BD3}).

In this foliated setting, $M$ is a complex manifold endowed with a holomorphic nonsingular foliation 
$\mathcal F$ and $E_H$ is a holomorphic principal $H$-bundle over $M$.

Define the foliated Atiyah bundle as the subbundle of the Atiyah bundle $\text{At}(E_H)$ defined by 
\begin{equation}\label{atF}
\text{At}_{\mathcal F}(E_H)\, :=\, ({d'p})^{-1}(T{\mathcal F})\,\subset\,
\text{At}(E_H)\, ,
\end{equation}

where $d'p$ was introduced in (\ref{c3}).
So from \eqref{c3} we get the short exact sequence
\begin{equation}\label{at2}
0\, \longrightarrow\, \text{ad}(E_H)\, \longrightarrow\,\text{At}_{\mathcal F}(E_H)\,
\stackrel{\mathrm{d}'\pi}{\longrightarrow}\, T{\mathcal F}\, \longrightarrow\, 0\, ,
\end{equation}
with $\mathrm{d}'\pi$ is the restriction of $d'p$
in \eqref{c3} to the subbundle $\text{At}_{\mathcal F}(E_H)$.

Recall that the principal $H$-bundle $E_H$ is said to admit {\it a partial holomorphic connection}\index{partial connection on a principal bundle} in the direction of $\mathcal F$ if the exact sequence 
(\ref{at2}) splits, meaning there exists a holomorphic homomorphism
$$
\lambda \, :\, T{\mathcal F}\, \longrightarrow\, \text{At}_{\mathcal F}(E_H)
$$
such that $\mathrm{d}'\pi\circ\lambda\,=\, \text{Id}_{T\mathcal F}$, where
$\mathrm{d}'\pi $ is the projection homomorphism in \eqref{at2}. The usual notion of a holomorphic connection over $M$ corresponds to the case where $\mathcal F$ is of codimension zero.

As it was first observed by Ehresmann the curvature of the partial connection $\lambda \, :\, T{\mathcal 
F}\, \longrightarrow\, \text{At}_{\mathcal F}(E_H)$ coincides with the obstruction to the integrability of 
the image of $\lambda$.

The connection $\lambda$ is
called {\it flat } if the image of $\lambda$ is integrable.

The foliated Atiyah bundle ${\rm At}_{\mathcal F}(E_H)$ being a subbundle of ${\rm At}(E_H)$, the partial connection 
$\lambda\, :\, {T\mathcal F}\, \longrightarrow\, \text{At}_{\mathcal F}(E_H)$ induces a unique
homomorphism $\lambda'\,:\,{T\mathcal F}\, \longrightarrow\, \text{At}(E_H)$ and from \eqref{c3} we get the following exact sequence
\begin{equation}\label{at3}
0\, \longrightarrow\, \text{ad}(E_H)\, \stackrel{\iota'}{\longrightarrow}\,
\text{At}(E_H)/\lambda'({T\mathcal F})\,
{\longrightarrow}\, TM/{T\mathcal F}\,
=\, {\mathcal N}_{\mathcal F}\, \longrightarrow\, 0\,.
\end{equation}

The following statement was proved in \cite{BD1} (see also Lemma 11.5.1 in \cite{BD3}).

\begin{lemma}[{\cite{BD1,BD3}}]\label{lemma1}
The flat partial connection $\lambda$ on $E_H$ induces a unique flat partial
connection along $\mathcal F$ on ${\rm At}(E_H)/\lambda'({T\mathcal F})$ such that the
homomorphisms in the exact sequence \eqref{at3} are partial connection preserving (where $\text{ad}(E_H)$ is endowed with the canonical connection induced from $(E_H,\lambda)$ and $TM/{T\mathcal F}$ is endowed with its canonical Bott connection).
\end{lemma}

{\it A transversal holomorphic Cartan geometry} with model $(G,\, H)$ on the foliated manifold $(M,\, 
{\mathcal F})$ is given by the following data

I) A holomorphic principal $H$-bundle $E_H$ on $M$ equipped with a flat partial connection $\lambda$ along 
$\mathcal F$;

II) A holomorphic homomorphism

\begin{equation}\label{beta}
\beta\,:\, \text{At}(E_H)/\lambda'(T{\mathcal F})\, \longrightarrow\,
\text{ad}(E_G)\, ,
\end{equation}

such that the following three conditions hold:
\begin{enumerate}
\item $\beta$ is partial connection preserving;

\item $\beta$ is an isomorphism;

\item the following diagram is commutative:
\begin{equation}\label{cg1}
\begin{matrix}
0 &\longrightarrow & \text{ad}(E_H) &\stackrel{\iota'}{\longrightarrow} &
\text{At}(E_H)/\theta'(T{\mathcal F}) &
\longrightarrow & {\mathcal N}_{\mathcal F} &\longrightarrow & 0\\
&& \Vert &&~ \Big\downarrow\beta && ~ \Big\downarrow\overline{\beta}\\
0 &\longrightarrow & \text{ad}(E_H) &\stackrel{\iota_1}{\longrightarrow} &
\text{ad}(E_G) &\longrightarrow &
\text{ad}(E_G)/\text{ad}(E_H) &\longrightarrow & 0
\end{matrix}
\end{equation}
\end{enumerate}
where the top exact sequence is the one in \eqref{at3}, while the bottom exact sequence is given by the canonical embedding $\iota_1$ of the adjoint bundle for $E_H$ into the adjoint bundle for $E_G$.

Notice that condition (1) in the above Definition is with respect to the partial
connection on $\text{At}(E_H)/\lambda'(T{\mathcal F})$ constructed in Lemma \ref{lemma1}
and the partial connection on $\text{ad}(E_G)$ inherited from the partial connection $(E_H,\lambda)$.

If the condition (2) in the above Definition is weakened to the condition (2') requiring that $\beta$ is an 
isomorphism over on an open dense set in $M$, we get the notion of transversal {\it branched} Cartan 
geometry as defined in \cite{BD1}. Moreover, if one completely drops the condition (2) in the above 
Definition, this leads to the more general notion of {\it a transversal generalized Cartan geometry} as 
defined in \cite{BD3}.

In this framework the following Lemma was proved in \cite{BD3}.

\begin{lemma}[\cite{BD3}] \label{lem2}
A transverse generalized holomorphic Cartan geometry with model $((E_H,\, \lambda),\, \beta)$ over the foliated manifold $(M, 
\mathcal F)$ defines a canonical holomorphic connection on the principal $G$-bundle $E_G$ which is flat in the direction of 
$\mathcal F$.
\end{lemma}

Denote by $\nabla$ the holomorphic connection constructed in Lemma \ref{lem2}. Denote by 
$$
{\rm Curv}(\nabla)\, \in\, H^0(M,\, \text{ad}(E_G)\otimes\Omega^2_M)
$$
its curvature.

The homomorphism $\beta$ being connection preserving, the contraction of ${\rm Curv}( \nabla)$ with any 
tangent vector to $T\mathcal F$ vanishes. This implies that ${\rm Curv}(\nabla)$ descends to a section of 
${\rm ad}(E_G)\otimes \bigwedge\nolimits^2 {\mathcal N}^*_{\mathcal F}$.

The transversal generalized Cartan geometry $((E_H, \lambda),\, \beta)$ is called \textit{flat} if the 
curvature tensor ${\rm Curv}(\nabla)$ vanishes identically.

Let us give here the following application of our results in \cite{BD1,BD3} and of Corollary \ref{cor1}.

\begin{theorem}\label{fol Cartan} 
Let $M$ be a compact K\"ahler manifold with trivial first Chern class. Suppose $M$ admits a nonsingular 
holomorphic foliation endowed with a transversal holomorphic Cartan geometry $(G,H)$, with $G$ being a 
complex affine algebraic group. Then the fundamental group of $M$ is infinite.
\end{theorem}

\begin{proof} 
Assume, by contradiction that the fundamental group of $M$ is finite. By replacing $M$ by its universal cover we may assume that $M$ is compact.

It was proved in \cite[Section 5.2]{BD1} that the transverse Cartan geometry is flat. This is also a 
consequence of Corollary \ref{cor1} here which implies that ${\rm Curv}(\nabla)$ vanishes (with $\nabla$ 
being the connection on $E_G$ defined in Lemma \ref{lem2}).

Since $M$ is simply connected, it follows that the developing map $\rm{dev} $ is a submersion from $M$ to 
$G/H$ (such that $T\mathcal F$ coincides with the kernel of the differential of $\rm{dev} $) 
\cite{BD1,BD3}. The fibers of this submersion must be also K\"ahler manifolds with trivial first Chern 
class. Indeed, for any fiber $F$ on the developing map we have that $c_1(F)$ coincides with the restriction 
to $F$ of $c_1(M)$.

Since $G/H$ is connected and the image of $\rm{dev}$ is open and closed (compact) it follows that $\rm{dev} 
(M)\,=\,G/H$ and $G/H$ is compact and belongs to Fujiki class $\mathcal C$ (as holomorphic image of a K\"ahler 
manifold) \cite{Fu}. Since $G/H$ is homogeneous, the Borel-Remmert theorem implies $G/H$ is a product of a 
rational manifold with a complex compact torus \cite{Fu}. Stein factorization theorem implies that the 
developing map factors through a holomorphic map $\rm{dev}' \,:\, M \longrightarrow
\, B'$ with connected fibers, where $B'$ 
is a finite unramified cover of $G/H$. Since $M$ is simply connected, $B'$ is also simply connected. This 
implies that the fundamental group of $G/H$ is finite. Hence the compact torus factor in the Borel-Remmert 
decomposition is trivial. Consequently, $G/H$ is a rational manifold.

In particular $G/H$ is projective and Theorem 3.1 in \cite{TZ} implies that $G/H$ must have trivial first Chern class: a 
contradiction.
\end{proof} 

\section*{Acknowledgements}

We thank the referee for helpful comments to improve the exposition.


\end{document}